\documentclass[a4paper,12pt]{article}
\usepackage{graphicx} 

\usepackage{tikz}

\usepackage{amsmath,amsthm,amscd,amssymb,eucal,mathrsfs}
\usepackage{amsfonts}
\usepackage{latexsym}
\usepackage{dsfont}
\usepackage[all]{xy}
\usepackage{fullpage}

\newtheorem{thm}{Theorem}[section]

\newtheorem{cor}[thm]{Corollary}
\newtheorem{remark}[thm]{Remark}

\newtheorem{Remark}[thm]{Remark}

\newtheorem{Proposition}[thm]{Proposition}

\newtheorem{Lemma}[thm]{Lemma}

\newtheorem{Ass}[thm]{Assumption}

\newcommand{\mathset}[1]{{\left\{#1\right\}}}
\newcommand{\absolute}[1]{\left\lvert#1\right\rvert}
\newcommand{\norm}[1]{\left\|#1\right\|}

\DeclareMathOperator{\supp}{supp}

\DeclareMathOperator{\closure}{cl}

\title{Boundary Value Problems for  $p$-Adic Elliptic Parisi-Z\'u\~niga Diffusion} 
\author{Patrick Erik Bradley}
\date{\today}

\begin{document}

\maketitle


\begin{abstract}
Elliptic integral-differential operators resembling the classical elliptic partial differential equations are defined over a compact $d$-dimensional $p$-adic domain together with associated Sobolev spaces relying on coordinate Vladimirov-type Laplacians dating back to an idea of Wilson Z\'u\~niga-Galindo in his previous work. The associated Poisson equations under boundary conditions are solved and their $L^2$-spectra are determined. Under certain finiteness conditions, a Markov semigroup acting on the Sobolev spaces which are also Hilbert spaces can be associated with such an operator and the boundary condition. It is shown that this also has an explicitly given heat kernel as an $L^2$-function, which allows a Green function to be derived from it.
\end{abstract}

\section{Introduction}

Elliptic partial differential equations over the $p$-adic numbers are much less studied than their classical counterparts. The latter can be learned about e.g.\ in \cite{evans}. Certain constructions like operators and Sobolev spaces can be carried over to $p$-adic domains. In most cases, this has been done for $p$-adic pseudodifferential operators, as e.g.\ in 
\cite{RZ2010,Zuniga2017,ZunigaZeta2017}. 
Alternative constructions of Sobolev spaces over more general abelian groups are found e.g.\ in \cite{GK2015,GK2020,GKR2014}.
 $p$-Adic Sobolev embedding theorems are proved \cite{Kim2010} and  \cite[Ch.\ III.7]{Taibleson1975}
or H\"older boundary regularity results \cite{Kochubei2023}.
However, so far the operators themselves are confined to special kinds of elliptic operators, built on  bounded versions of Vladimirov-Taibleson Laplacians, cf.\ \cite{LQ2016}.
\newline

Concerning the $p$-adic spaces on whose functions the operators are to act, other than the additive groups of local fields, there are under consideration 
$p$-adic multivariate pseudo\-differential operators
  \cite{RW2023,BGPW2014,TAB2024}
  on $\mathds{Q}_p^d$. In these cases, the interest lies on the underlying stochastic process which happens to have the Markov property. But also
 $p$-adic vector fields on $p$-adic Lie groups, building on these previous ideas  are considered \cite{EngelGroup,UnitaryDual_p}.
\newline

Since a large amount of the work on $p$-adic differential-integral operators focuses on the induced diffusion, many authors study the Markov process associated with it, culminating so far in scaling limit theorems using path spaces
\cite{PRSWY2024,Weisbart2024}.
This inspired the author to study diffusion invariant under finitely generated discrete groups which give rise to $p$-adic Riemann surfaces aka Mumford curves
\cite{brad_thetaDiffusionTateCurve,SchottkyDiff} after observing that $p$-adic Laplacians can reconstruct finite graphs \cite{BL_shapes_p}.
Finally, a successful approximation approach for Green functions 
on manifolds via their relationship with heat kernels 
\cite{Ding2002} motivates to prove the existence of Markov processes, heat kernels and Green functions on a compact subspace of
$\mathds{Q}_p^d$ for elliptic operators which resemble classical elliptic partial differential operators and their applications to suitable Sobolev spaces under the additional condition of vanishing on a boundary of an open subset of the compact domain induced by the operator's kernel functions.
This is a different approach from \cite{NonAutonomousDiffusion}, where $p$-adic Dirichlet and von Neumann boundary conditions were defined.
\newline

The results of this article can be summarised as follows:
\newline

After summarising the spectra and the Feller semigroup property of component Laplacians $\mathcal{L}_i$ for $i=1,\dots,d$, defined as in \cite{BL-TopoIndex_p}, these are used to define elliptic operators $P(\mathcal{L})$ as polynomials in degree $2$ whose coefficients are $L^\infty$-functions and the unknowns replaced with the component Laplacians, such that the coefficent matrix in the homogeneous degree $2$ part is positive definite almost everywhere and has a positive infimum eigenvalue. First, the case of constant coefficients is treated, resulting in an explicit expression for its $L^2$-spectrum in Theorem \ref{L2-Spectrum}.
\newline 

After this, boundary conditions and Sobolev spaces are introduced. The former depend on the kernel function in that it expresses the transitions between vertices of a finite graph in each coordinate, whose vertices are represented as disjoint $p$-adic discs, and a Vladimirov-type diffusion inside each of these discs. 
A $p$-adic divergence property  is proved in Theorem \ref{divergence_p}.
\newline

Energy estimates written out in Theorems \ref{energyHomogeneous} and \ref{energy} prepare the way to solving the Poisson equation given by the elliptic operator $P(\mathcal{L})$. It has  weak solutions in the underlying Sobolev spaces, as shown in Corollaries \ref{WeakSolutionDeg2} and
\ref{weak_bvp_perturbed}.
\newline

If the coefficients of $P(\mathcal{L})$ are test functions, then it is unitarily diagonalisable with point spectrum under  certain invariance and commutativity conditions, as expressed in Theorem \ref{Pdiagonalisable}. If in this case, the operator is elliptic, then there is an associated
 Markov semigroup $e^{-tP(\mathcal{L})}$ for $t\ge0$ acting on suitable Sobolev spaces.
This is Theorem \ref{MarkovSemigroup}. In this case, there is also a 
heat kernel function of $L^2$-type: Theorem \ref{heatKernelFunction4Semigroup}, and a  Green function is obtained by solving the corresponding Poisson equation: Corollary \ref{GreenFunctionExists}.
\newline

The following Section 2 introduces the coordinate Laplacians as so-called Z\'u\~niga-Parisi operators and recalls their spectral and stochastic properties. Also, sub-Laplacians of arbitrary order are defined. Their $L^2$-spectrum is studied in Section 3. Section 4 introduces boundary conditions and Sobolev spaces. Elliptic operators in the form of divergence operators are introduced in Section 5, where also the corresponding Poisson equations are studied, and the unitary diagonalisability together with their spectra as point spectra are proved. Section 6 proves the Markov property of the semigroup action on Sobolev spaces associated with $P(\mathcal{L})$ and proves the convergence properties of the associated heat kernel function as well as the Green function.
\newline

Throughout this article, only the Sobolev spaces which are also Hilbert spaces are actually used.

\section{Z\'u\~niga-Parisi sub-Laplacian operators}

Let $\mathds{Q}_p$ be the field of $p$-adic functions, and
let $F\subset\mathds{Q}_p^d$ be a compact open subset. Let $\pi_i\colon\mathds{Q}_p^d\to\mathds{Q}_p$ be the projection onto the $i$-th coordinate, and
fix a disjoint covering $\mathcal{U}_i$ of $\pi_i(F)$ as
\[
\pi_i(F)=\bigsqcup\limits_{k=1}^{N_i}
B_{i,k}
\]
with $p$-adic discs $B_{i,k}\in\mathcal{U}_i$ for every $i=1,\dots,d$. This is possible, because the projection maps $\pi_i$ are continuous and open. Using multi-index notation, this yields a disjoint covering 
\[
\mathcal{U}=\mathcal{U}_1\times\dots\times\mathcal{U}_d
\]
of $F$ given by
\[
F=\bigsqcup\limits_{\underline{k}\in\underline{N}}B_{\underline{k}}
\]
with polydiscs
\[
    B_{\underline{k}}=\prod\limits_{i=1}^d
    B_{i,k_i}\in\mathcal{U}
\]
for
\[
\underline{k}=(k_1,\dots,k_d)\in\underline{N}=\prod\limits_{i=1}^d\mathset{1,\dots,N_i}
\]
and $N_1,\dots,N_d\in\mathds{N}$.
\newline

The Haar measure $dx$ on $\mathds{Q}_p^d$ is the product measure 
\[
dx=dx_1\wedge\dots\wedge dx_d
\]
normalised such that
\[
\int_{\mathds{Z}_p^d}dx=1
\]
and is denoted on sets as
\[
\mu(A)=\int_A\,dx\,.
\]
For a polydisc $B_{\underline{k}}$ it holds true that
\[
\mu\left(B_{\underline{k}}\right)
=\prod\limits_{i=1}^d\mu_i(B_{k_i})
=p^{-(k_1+\dots+k_d)}\,,
\]
where
$k=(k_1,\dots,k_d)\in\mathds{Z}^d$, and where $\mu_i$ is the $i$-th component Haar measure $dx_i$ for $i=1,\dots,d$. It is also normalised such that the $i$-th component unit disc has measure one for $i=1,\dots,d$.

\subsection{Component hierarchical Parisi Operators}
\label{sec:componentOP}

Using the notation
\[
U(z)\in\mathcal{U},\quad U_i(\zeta_i)\in\mathcal{U}_i
\]
for the unique polydisc in $\mathcal{U}$ containg $z=(\zeta_1,\dots,\zeta_d)\in F$, and for the unique disc in $\mathcal{U}_i$ containing $\zeta_i\in\pi_i(F)$, it is now possible to define an $i$-th component Laplacian $\mathcal{L}_{X_i,\alpha_i}$ on functions $f\colon F\to\mathds{C}$ as
\begin{align}\label{componentOp}
\mathcal{L}_{X_i,\alpha_i}f(x)
=\int_{\pi_i(F)}L_i(\xi_i,\eta_i)(f(x)-f(\xi_1,\dots,\eta_i,\dots,\xi_d))\,d\eta_i
\end{align}
with $\alpha_i>0$, $x=(\xi_1,\dots,\xi_d)\in F$, and
\[
L_i(\xi_i,\eta_i)=
\begin{cases}
\absolute{\xi_i-\eta_i}_p^{-\alpha_i},&U_i(\xi_i)=U_i(\eta_i),\;\xi_i\neq\eta_i
\\
w_i(U_i(\xi_i),U_i(\eta_i)),&U_i(\xi_i)\neq U_i(\eta_i)
\end{cases}
\]
with
\[
w_i(U_i(\xi_i),U_i(\eta_i))\ge0
\]
symmetric on $\mathcal{U}_i\times \mathcal{U}_i$ outside the diagonal. The operator $\mathcal{L}_{X_i,\alpha_i}$ defines a hierarchical Parisi Laplacian operator in the terminology of \cite{BL-TopoIndex_p}, and is the $i$-th component Laplacian, where $X_i$ indicates the $i$-th coordinate in $\mathds{Q}_p^d$.
\newline

\begin{Lemma}\label{SchwartzLemma}
It holds true that 
\[
\mathcal{L}_{X_i,\alpha_i}\circ\mathcal{L}_{X_j,\alpha_j}=\mathcal{L}_{X_j,\alpha_j}\circ\mathcal{L}_{X_i,\alpha_i}
\]
for $i,j=1,\dots,d$
on the space $\mathcal{D}(F)$ of locally constant functions on $F$.
\end{Lemma}

\begin{proof}
Assume that $i\neq j$. Then
it holds true that
\begin{align*}
\mathcal{L}_{X_i,\alpha_i}&\mathcal{L}_{X_j,\alpha_j}f(x)
\\
&=\int_{\pi_i(F)}L_i(\xi_i,\eta_i)(\mathcal{L}_{X_j,\alpha_j}f(x))-\mathcal{L}_{X_j,\alpha_j}f(\xi_1,\dots,\eta_i,\dots,\xi_d)\,d\eta_i
\\
&=\int_{\pi_i(F)}L_i(\xi_i,\eta_i)
\left(\int_{\pi_j(F)}L_j(\xi_j,\eta_j)
(f(x)-f(\xi_1,\dots,\eta_j,\dots,\xi_d))\,d\eta_j
\right.
\\
&\left.-\int_{\pi_j(F)}
(f(\xi_1,\dots,\eta_i,\dots,\xi_d)-f(\xi_1,\dots,\eta_i,\dots,\eta_j,\dots,\xi_d))
\,d\eta_j\right)\,d\eta_i
\\
&=\int_{\pi_i(F)}\int_{\pi_j(F)}
L_i(\xi_i,\eta_i)L_j(\xi_j,\eta_j)
(f(x)-f(\dots,\eta_j\dots)
\\
&-f(\dots,\eta_i,\dots)+f(\dots,\eta_i,\dots,\eta_j,\dots))\,d\eta_i\,d\eta_j\,,
\end{align*}
which, using Fubini's Theorem and the symmetry of the kernel functions $L_i(\xi_i,\eta_i)$ and $L_j(\xi_j,\eta_j)$, implies the assertion.
\end{proof}

The following operator will now be used:
 \begin{align}\label{pushforward}
\pi_{i,*}\mathcal{L}_i f(\xi_i)=\int_{\pi_i(F)}L_i(\xi_i,\eta_i)(f(\xi_i)-f(\eta_i))\,d\eta_i\,,
 \end{align}
 for $i=1,\dots,d$. The Kozyrev wavelets supported in $\pi_i(F)$ will be denoted as
 \begin{align}\label{KozyrevWavelet}
\psi_{B_n(a),j}=p^{\frac{n}{2}}\chi(p^{-n-1}j\xi_i)\Omega(\xi_i\in B_n(a))\,,
 \end{align}
 where $B_n(a)$ is its support with $a\in \pi_i(F)$, $j=1,\dots,p-1$,
 and $\chi\colon\mathds{Q}_p\to S^1$ is a fixed unitary character. Notice that a Kozyrev wavelet has the three parameters $n\in\mathds{Z}$, $a\in\mathds{Q}_p$, and $j\in\mathset{1,\dots,p-1}$.
 \newline

 Already Z\'u\~niga observed in \cite{ZunigaNetworks} that his operators have two types of eigenfunctions: Kozyrev wavelets on the one hand, and functions which are linear combinations of indicators of maximal discs in the compact set $K$ of $\mathds{Q}_p$ on which the operator lives. We will call the latter \emph{graph eigenfunctions}. This classification of eigenfunctions also holds true for the more general operators defined in \cite{BL-TopoIndex_p}. For the
 operator $\pi_{i,*}\mathcal{L}_{X_i,\alpha_i}$, this is also the case:

\begin{thm}[Pushforward-component operator Spectrum]\label{ComponentSpectrum}
  The Hilbert space $L^2(\pi_i(F),\mu_i)$ has an orthonormal eigenbasis for $\pi_{i,*}\mathcal{L}_{X_i,\alpha_i}$ consisting of the Kozyrev wavelets $\psi_{B_n(a),j}$, $j=1,\dots,p-1$, supported in $B_n(a)\subset \pi_i(F)$, and associated graph eigenfunctions. The eigenvalue associated with $\psi=\psi_{B_n(a),j}$ is
    \[
    \lambda_\psi=  p^{n(1+\alpha_i)}(p^{-m(1+\alpha_i)}+1)
    + \sum\limits_{U_i(b)\neq U_i(a)}w_i(U_i(a),U_i(b))\mu_i(U_i(b))-1\,,
\]
where it is assumed that $U_i(a)=B_m(a)$, and contains $B_n(a)$.
The operator is self-adjoint, positive semi-definite, and each eigenvalue has only finite multiplicity.
\end{thm}

\begin{proof}
Cf.\ \cite[Theorem 3.6]{BL-TopoIndex_p}. The idea behind that proof is actually a simplification of the idea  behind the proof of \cite[Theorem 4.10]{SchottkyDiff}.
\end{proof}

\begin{Remark}\label{subsetting}
If $F$ is replaced by an open subset $U\subseteq F$, then Theorem \ref{ComponentSpectrum} remains valid with the induced graph structure on the subset of vertices represented by open sets in $\pi_i U$ given as the intersections of the discs representing the original graph with $U$. Cf.\ \cite[Theorem 4.10]{SchottkyDiff}, where this situation was studied under Schottky invariance. 
\end{Remark}

For later reference, include the following result:

\begin{thm}[Component  Feller Semigroup]\label{ComponentFeller}
  There exists a probability measure $p_t(x,\cdot)$ with $t\ge0$, $x\in F$, on the Borel $\sigma$-algebra of $\pi_i(F)$ such that the Cauchy problem for the heat equation
\[
\frac{\partial}{\partial t}u(x,t)+\mathcal{L}_{X_i,\alpha_i}u(x,t)=0
\]
for $\alpha_i>0$ has a unique solution in $C^1((0,\infty),C(F))$ of the form
\[
u(x,t)=\int_{\pi_i(F)}L_i(\xi_i,\eta_i)p_t(x,d\eta_i)
\]
In addition, $p_t(x,\cdot)$ is the transition function of a strong Markov process whose paths are c\`adl\`ag.
\end{thm}

\begin{proof}
The proof is analogous to that of \cite[Theorem 3.5]{BL-TopoIndex_p}, which is an adaptation of the proof of \cite[Lemma 5.1]{SchottkyDiff}.
\end{proof}

\begin{Remark}\label{subsettingFeller}
Theorem \ref{ComponentFeller} remains true, if $F$ is replaced by an open subset $U\subseteq F$, similarly as with the $L^2$-spectrum of Theorem \ref{ComponentSpectrum}. Cf.\ \cite[Theorem 5.2]{SchottkyDiff}, where this was shown under Schottky invariance.
\end{Remark}

\subsection{Sub-Laplacians of arbitrary order}

The operator from which the operators of interest in this article are built, is the tuple
\[
\mathcal{L}=(\mathcal{L}_1,\dots,\mathcal{L}_d)\,,
\]
where $\mathcal{L}_i$ is the $i$-th operator $\mathcal{L}_{X_i,\alpha_i}$ defined
in (\ref{componentOp}). The operator $\mathcal{L}_i$ acts on  real- or complex-valued functions on $F$. 
\newline

Following ideas from \cite{EngelGroup,UnitaryDual_p},
first define the following sub-Laplacian 
\begin{align}\label{subLaplacian}
P_1(\mathcal{L})=\sum\limits_{i=1}^d\gamma_i\mathcal{L}_{X_i,\alpha_i}
\end{align}
with $\gamma_1,\dots,\gamma_d\colon F\to\mathds{R}$
suitable functions. The operator acts on functions $f\colon F\to\mathds{C}$. This integral operator can be viewed as a $p$-adic analogon of a partial differential operator in classical analysis. Observe that $P_1(\mathcal{L})$ is a symmetric operator, because all component Laplacians are symmetric. Other kinds of $p$-adic analoga of partial differential operators are work in progress, e.g.\ such which are built from advection-type operators as in \cite{divgrad_p}.
\newline

Observe that if $\gamma_1,\dots,\gamma_d$ are constant, then $P_1(\mathcal{L})$ is an integral operator of the form
\begin{align}\label{IntegralLaplacian}
P_1(\mathcal{L})f(x)=\int_F L(x,y)(f(x)-f(y))\,dy
\end{align}
with
\begin{align}\label{subLaplacianKernel}
L(x,y)=\sum\limits_{i=1}^d\gamma_iL_i(\xi_i,\eta_i)\prod\limits_{j=1\atop j\neq i}^d\delta_{\xi_i}(\eta_i)
\end{align}
for $x=(\xi_1,\dots,\xi_d),y=(\eta_1,\dots,\eta_d)\in F$, and where $\delta_{\xi_i}$ is the delta-function on $\pi_i(F)$ supported in $\xi_i$ for $i=1,\dots,d$.
\newline

Using Lemma \ref{SchwartzLemma}, it is possible to take a polynomial $P(X_1,\dots,X_d)\in\mathds{C}[X_1,\dots,X_d]$, and construct the operator
\[
P(\mathcal{L})=P(\mathcal{L}_{X_1,\alpha_1},\dots,\mathcal{L}_{X_d,\alpha_d})
=\sum\limits_{\underline{k}\in\mathds{N}^d}\gamma_{\underline{k}}\mathcal{L}_{X,a}^{\underline{k}}
\]
 for 
\[
a=(\alpha_1,\dots,\alpha_d)\in\mathds{R}_{>0}
\]
acting on $\mathcal{D}(F)$, where polynomial $P$ is given as
\[
P(X)=\sum\limits_{\underline{k}\in\mathds{N}^d}\gamma_{\underline{k}}X^{\underline{k}}\in\mathds{C}[X]
\]
in multi-index notation for the variable tuple $X=(X_1,\dots,X_d)$. The operator (\ref{subLaplacian}) is also such an operator, determined by a linear polynomial $P_1\in\mathds{C}[X_1,\dots,X_n]$.
\newline

Again, $P(\mathcal{L})$ is an integral operator  whose kernel function $L(x,y)$ can be obtained by iterating the calculation in the proof of Lemma \ref{SchwartzLemma}. Notice that from this, it can be seen that the integral operator  uses higher differences of the function $f(x)$.
\newline

\begin{remark}
  It is also possible to take the coefficients $\gamma_{\underline{k}}$ appearing in operator $P(\mathcal{L})$ as functions  $\gamma_{\underline{k}}\colon F\to\mathds{C}$, and thus produce a very general linear operator similar to partial differential operators in the classical case. Such operators will be studied from Section \ref{EllipticDivergenceOp} on in the case of degree $2$.
 In view of the following section, sub-Laplacians w.r.t.\ a polynomial $P(X)$ of degree $2$ are a special case of this kind of general operators.
\end{remark}

\section{$L^2$-Spectrum of $p$-adic sub-Laplacians}

The aim is to construct an explicit eigenbasis of $L^2(F)$ for a given sub-Laplacian $P(\mathcal{L})$ with $P\in\mathds{C}[X_1,\dots,X_d]$,
using the decomposition
\[
L^2(F)=\bigotimes\limits_{i=1}^d
\left(L^2(\pi_i(F))_0\oplus\mathds{C}^{N_i}\right)\,,
\]
where
\[
L^2(\pi_i(F))_0=\mathset{f\in L^2(\pi_i(F))\mid\int_{\pi_i(F)}f(\xi_i)\,d\xi_i=0}
\]
and $N_i$ is the cardinality of $\mathcal{U}_i$.
This will turn out useful for the  case of the  elliptic operators defined below in Section \ref{EllipticDivergenceOp}.
\newline

Notice that the decomposition
\[
L^2(\pi_i(F))_0\oplus \mathds{C}^{N_i}
\]
is an invariant decomposition of a $p$-adic Laplacian in the case $d=1$, which was shown in Theorem \ref{ComponentSpectrum}.
More precisely, the part $\mathds{C}^{N_i}$ is spanned by eigenfunctions $\varphi_i$ of the Laplacian
associated with the adjacency matrix
\[
\left(\mu(U_i)w(U_i,V_i)\right)_{U_i,V_i\in\mathcal{U}_i}
\]
of a simple graph $G_i$ on $N_i$ vertices (i.e.\ it is assumed that $w_i(U_i,U_i)=0$), called the \emph{$i$-th component graph} of $\mathcal{L}_{X_i,\alpha_i}$,
whereas the part $L^2(\pi_i(F))_0$ is spanned by Kozyrev wavelets $\psi_i$ supported in $\pi_i(F)$ for $i=1,\dots,d$. Both, $\varphi_i$ and $\psi_i$ 
 are eigenfunctions of the operator $\Phi_i=P(1,\dots,\mathcal{L}_{X_i,\alpha_i},\dots,1)$ projected down to an operator acting on $L^2(\pi_i(F))$, where
 $P(X)\in\mathds{C}[X]=\mathds{C}[X_1,\dots,X_d]$ is the defining polynomial of $P(\mathcal{L})$. Hereby, the  operator $\Phi_i$ is a linear combination of tensor products of powers of 
 $\pi_{i,*}\mathcal{L}_i$, as defined in (\ref{pushforward})
with the identity operator $1$ on the other component spaces $L^2(\pi_j(F))$ for $j\neq i$.
The eigenvalue $\lambda_{\psi_i}$ associated with a Kozyrev wavelet $\psi_i$ can be calculated using Theorem \ref{ComponentSpectrum}.

\begin{thm}\label{L2-Spectrum}
The space $L^2(F)$ has an orthonormal eigenbasis for $P(\mathcal{L})$ consisting of functions of the form
\[
b(x)=\prod\limits_{i=1}^db_i(\xi_i),\;
\]
where $x=(\xi_1,\dots,\xi_d)\in F$, and $b_i(\xi_i)$ is either a Kozyrev wavelet $\psi_i$ supported in $\pi_i(F)$, or a Laplacian eigenfunction $\varphi_i$ for the weighted $i$-th coordinate graph $G_i$. The corresponding eigenvalue $\lambda_b$ equals
\[
\lambda_b=P(\lambda_{b_1},\dots,\lambda_{b_d})\,
\]
where  $\lambda_{b_i}$ is the eigenvalue associated with $b_i(\xi_i)$ for $i=1,\dots,d$. The operator $P(\mathcal{L})$ is self-adjoint.
\end{thm}

\begin{proof}
In view of the remark in the paragraph before the Theorem,
this is an immediate consequence of applying
Theorem \ref{ComponentSpectrum}
to the tensor product space, and by the construction of operator $\mathcal{L}$.
\end{proof}

\begin{cor}\label{heatKernelFunction}
  Given $P(X)\in\mathds{R}[X_1,\dots,X_d]$ such that $P(\mathcal{L})$ has only non-negative eigenvalues, each having only finite multiplicity, there exists an associated heat kernel function
  providing  a fundamental solution for the heat equation
  \[
\frac{\partial}{\partial t}u(x,t)+P(\mathcal{L})u(x,t)=0
\]
with $u(x,0)=u(x)\in L^2(F,\mu)$ for $t\ge0$.
\end{cor}

\begin{proof}
  The prospective heat kernel function is given by
  \[
p(t,x,y)=\sum\limits_b e^{-\lambda_b t}b(x)b(y)\,,
\]
where $b$ runs through the eigenbasis of Theorem \ref{L2-Spectrum}. In order to show that this sum converges for $x,y\in F$, assume first that $x\neq y$. In this case, there are only finitely many functions $b$ of the eigenbasis such that $b(x)b(y)\neq 0$. Since each eigenvalue has only finite multiplicity, the convergence now follows also if $x=y$. This proves the assertion.
\end{proof}

An example for Corollary \ref{heatKernelFunction} is given by a polynomial $P(X)$ of degree $2$ such that the coefficients of the homogeneous part $P_2(X)$ form a positive definite matrix. This is a special case of an elliptic operator which will be dealt with in Section \ref{EllipticDivergenceOp}.


\section{Boundary Conditions and Sobolev Spaces}

Boundary conditions w.r.t.\ the  operators $\mathcal{L}_i$ are defined, as well as corresponding Sobolev spaces are introduced.

\subsection{Boundary conditions}

Define the $i$-th component boundaries of an open subset $U\subset F$ as
\begin{align*}
\delta_i U&=\mathset{\eta_i\in\pi_i(F\setminus U)\mid\exists \xi_i\in \pi_i(U)\colon L_i(\xi_i,\eta_i)\neq 0}
\end{align*}
and the $i$-th component boundary condition for a function $u\colon F\to\mathds{R}$ as
\[
u|_{\delta_i U}(x):=u(x)\int_{\delta_i U}L_i(\xi_i,\eta_i)\,d\eta_i=0
\]
for all $x=(\xi_1,\dots,\xi_d)\in U$. 
\newline

The boundary $U$ w.r.t.\ $\mathcal{L}$ is defined as
\[
\delta U=\bigsqcup
\limits_{i=1}^d (U_1\sqcup\delta_1 U)\times\dots\times
\delta_i U\times\dots\times (U_d\sqcup\delta_d U)\subset F\,,
\]
and we say that $u$ vanishes on the boundary $\delta U$, if
\[
u|_{\delta_iU}(x)=0
\]
for all $i=1,\dots,d$ and all $x\in U$. In this case, also write
\[
u|_{\delta U}(x)=0
\]
for all $x\in U$. Write also
\[
\closure_\delta U:=U\sqcup \delta U
\]
for the $\delta$-closure of $U$. It is determined by the connectivity structure of $F$ imposed by the kernel functions $L_1,\dots,L_d$. 

\begin{Lemma}\label{clopen}
The set $\closure_\delta U$ is  closed-open in $F$.
\end{Lemma}

\begin{proof}
This follows immediately by construction.
\end{proof}

Use the notation
\[
\mathcal{L}_{i,\phi}=\mathcal{L}_{X_i,\alpha_i;\phi}\,,
\]
i.e.
\[
\mathcal{L}_{i,\phi}u(x)=
\int_{\pi_i(F)}L_i(\xi_i,\eta_i)(u(x)-u(\xi_1,\dots,\eta_i,\dots,\xi_d))\phi(\xi_1,\dots,\eta_i,\dots,\xi_d)\,d\eta_i\,.
\]
There is a  Leibniz-like rule: 

\begin{Lemma}\label{LeibnizLike}
It holds true that
\[
\mathcal{L}_i(u\phi)=\mathcal{L}_{i,\phi}u+u\mathcal{L}_i\phi
\]
for $i=1,\dots,d$, whenever the two terms on the right-hand side are defined.
\end{Lemma}

\begin{proof}
Use the notation
\[
g(\hat{\eta_i}):=g(\xi_1,\dots,\eta_i,\dots,\xi_d)
\]
and calculate
\begin{align*}
\mathcal{L}_i(u\phi)(x)&=
\int_{\pi_i(F)}L_i(\xi_i,\eta_i)(u(\hat{\xi_i})\phi(\hat{\xi_i})-u(\hat{\eta_i})\phi(\hat{\eta_i}))\,d\eta_i
\\
&=\int_{\pi_i(F)}L_i(\xi_i,\eta_i)
(u(\hat{\xi_i})\phi(\hat{\xi_i})
-u(\hat{\xi_i})\phi(\hat{\eta_i})+u(\hat{\xi_i})\phi(\hat{\eta_i})
-u(\hat{\eta_i})\phi(\hat{\eta_i})
)\,d\eta_i
\\
&=
u(\hat{\xi_i})\int_{\pi_i(F)}L_i(\xi_i,\eta_i)(\phi_i(\hat{\eta_i})-\phi(\hat{\xi_i}))\,d\eta_i
\\
&+\int_{\pi_i(F)}L_i(\xi_i,\eta_i)(u(\hat{\xi_i})-u(\hat{\eta_i}))\phi(\hat{\eta_i})\,d\eta_i
\\
&=u(x)\mathcal{L}_i\phi(x)+
\mathcal{L}_{i,\phi}u(x)\,,
\end{align*}
which implies the assertion.
\end{proof}

\subsection{Sobolev spaces}

Let $U\subseteq F$ be an open subspace.
Define the  following Sobolev spaces:
\begin{align*}
  W^{k,q}(U)&=\mathset{f\in L^1(U)\mid\forall\underline{\ell}\in\mathds{Z}^d\colon\absolute{\underline{\ell}}\le k\;\Rightarrow\;\norm{\mathcal{L}^{\underline{\ell}}f}_{L^q(U)}<\infty}
  \\
  W^{k,q}_0(U)&=\mathset{f\in W^{k,q}(\closure_\delta U)\mid f|_{\delta U}=0}
\end{align*}
for $q>0$, where $\closure_F$ is the closure operator on subsets of $F$,
and with
\[
\mathcal{L}^{\underline{\ell}}:=\mathcal{L}_{X,a}^{\underline{\ell}}=\mathcal{L}_{X_1,\alpha_1}^{\ell_1}\cdots\mathcal{L}_{X_d,\alpha_d}^{\ell_d}
\]
with $\underline{\ell}\in\mathds{Z}^d$.
The Sobolev norm on $W^{k.q}(U)$ is defined as
\[
\norm{f}_{W^{k,q}(U)}=\left(\sum\limits_{\absolute{\underline{\ell}}\le k}
\norm{\mathcal{L}^{\underline{\ell}}f}_{L^q(U)}^q\right)^{\frac{1}{q}}
\]
just like in the classical case.

\begin{Proposition}\label{SobolevClosedSubspace}
The Sobolev spaces $W^{k,q}(U)$ are Banach spaces for $1\le q<\infty$ and $k\ge0$, and the space $W^{k,q}_0(U)$ is a closed subspace of $W^{k,q}(\closure_\delta U)$.
\end{Proposition}

\begin{proof}
  Following the proof of \cite[Theorem 5.2.2]{evans}, first observe that $\norm{\cdot}_{W^{k,q}(U)}$ is indeed a norm on $W^{k,q}(U)$. In order to see completeness, let $u_n\in W^{k,q}(U)$ be a Cauchy sequence. Then $\mathcal{L}^{\underline{\ell}}u_n$ is a Cauchy sequence in $L^q(U)$ for each $\absolute{\underline{\ell}}\le k$. Thus, there exists a sequence of functions $f_{\underline{\ell}}\in L^q(U)$ such that
  \[
\mathcal{L}^{\underline{\ell}}u_n\to
  f_{\underline{\ell}}
  \]
  in $L^2(U)$ for each $\absolute{\underline{\ell}}\le k$. In particular, it holds true that
  $u:=f_{(0,\dots,0)}\in L^q(U)$ is the limit of $u_n$. In order to now see that $u\in W^{k,q}(U)$ and $\mathcal{L}^{\underline{\ell}}u=f_{\underline{\ell}}$ for $\absolute{\underline{\ell}}\le k$, observe that for any $\phi\in\mathcal{D}(U)$, it holds true that
  \begin{align*}
\int_U\mathcal{L}^{\underline{\ell}}u\phi\,dx&=
    \int_U u\mathcal{L}^{\underline{\ell}}\phi\,dx=\lim\limits_{n\to\infty}
    \int_Uu_n\mathcal{L}^{\underline{\ell}}\phi\,dx
    \\
   & =\lim\limits_{n\to\infty}\int_U\mathcal{L}^{\underline{\ell}}u_n\phi\,dx
    =\int_Uf_{\underline{\ell}}\phi\,dx\,,
  \end{align*}
  where the self-adjointness of $\mathcal{L}^{\underline{\ell}}$ has been used, cf.\ Theorem \ref{L2-Spectrum}. The first assertion now  follows.
  It is also valid for $W^{k,q}(\closure_\delta U)$.
  Since $W^{k,q}(U)_0$ is obtained by restriction to an open subspace $U\subseteq\closure_\delta U$, the second assertion also follows. This proves the proposition.
\end{proof}

\begin{cor}\label{SobolevHilbertSpace}
  The Sobolev space $W^{k,2}(U)$ is a Hilbert space for $k\in\mathds{N}$.
\end{cor}

\begin{proof}
  The Sobolev norm on $W^{1,2}(U)$ takes the form
  \[
  \norm{f}_{W^{1,2}(U)}^2=
  \sum\limits_{\underline{\ell}\le k}
  \norm{\mathcal{L}^{\underline{\ell}}f}_{L^2(U)}^2
\]
  which comes from a suitable inner product on $W^{k,2}(U)$ for $k\in\mathds{N}$.
  This proves the assertion.
  \end{proof}

\begin{Proposition}[Poincar\'e Inequality]\label{PoincareInequality}
Let $u\in W^{1,2}(U)$ for $1\le q\le\infty$. Then there exists some $C>0$ such that
\[
\norm{u}_{L^2}\le C\norm{\mathcal{L}_iu}_{L^2}
\]
for $i=1,\dots,d$.
\end{Proposition}

\begin{proof}
  Using Theorem \ref{ComponentSpectrum},
  it follows from Theorem \ref{L2-Spectrum} that the function $u$ has an  expansion over an orthonormal  eigenbasis $\psi$ w.r.t.\ $\mathcal{L}_i$
as
\[
u=\sum\limits_{\psi}\alpha_\psi\psi
\]
and the eigenvalues $\lambda_\psi$ associated with $\psi$ are unboundedly increasing with shrinking support of the wavelet eigenfunctions in the $i$-th coordinate. Hence,
\[
\norm{u}_{L^2}^2=\sum\limits_{\psi}\absolute{\alpha_\psi}^2\le C
\sum\limits_\psi\absolute{\alpha_\psi}^2\absolute{\lambda_\psi}^2=
\norm{\mathcal{L}_iu}_{L^2}^2
\]
for some $C>0$, as asserted.
\end{proof}

\begin{Lemma}\label{boundaryCondDist}
Let $u\in W^{1,2}(\closure_\delta U)$. The boundary condition $u|_{\delta U}=0$ in the distributional sense is equivalent with
\[
\int_U\mathcal{L}_{i}(u\phi)(x)\,dx=0
\]
for all $\phi\in \mathcal{D}(U)$ and all $i=1,\dots,d$.
\end{Lemma}

\begin{proof}
It holds true that
\begin{align*}
\int_U\mathcal{L}_{i}&(u\phi)(x)\,dx
=\langle u\phi,\mathcal{L}_{i}1_U\rangle
\\
&=\int_U\int_{\pi_i(F\setminus U)}u(x)L_i(\xi_i,\eta_i)\,d\eta_i\,\phi(x)\,dx
\\
&=\int_U u(x)\int_{\delta_iU}L_i(\xi_i,\eta_i)\,d\eta_i\,\phi(x)\,dx
\\
&=\int_U u|_{\delta_i U}(x)\phi(x)\,dx\,,
\end{align*}
whose vanishing is equivalent with the boundary condition for $u$ in the distributional sense. This proves the assertion.
\end{proof}

Observe that, if $u|_{\delta_i U}=0$ in the distributional sense, it follows that 
\begin{align*}
\int_U\mathcal{L}_{i,\phi}u(x)\,dx
=-\int_Uu(x)\mathcal{L}_i\phi(x)\,dx
=-\langle u,\mathcal{L}_i\phi\rangle
=-\langle\mathcal{L}_iu,\phi\rangle
\end{align*}
for $\phi\in \mathcal{D}(U)$.
\newline

Define the operator $\pi_{i,*}\mathcal{A}_i\colon \mathcal{D}(F)\to\mathcal{D}(\pi_i(F))$ on functions $u\colon F\to\mathds{R}$ as
\[
\left[(\pi_{i,*}\mathcal{A}_i)u\right](\eta_i)
=\int_UL_i(\xi_i,\eta_i)u(x)\,dx
\]
for 
$x=(\xi_1,\dots,\xi_d)\in U$, $\eta_i\in F$, and $i=1,\dots,d$. This allows to imitate a divergence theorem as follows:

\begin{thm}[$p$-adic Divergence Theorem]\label{divergence_p}
It holds true that
\[
\int_U \mathcal{L}_if(x)\,dx
=\int_{\delta_iU}\left[(\pi_{i,*}\mathcal{A}_i)f\right](\eta_i)
\,d\eta_i
\]
for $i=1,\dots,d$.
\end{thm}

\begin{proof}
It holds true that
\begin{align*}
\int_U\mathcal{L}_if(x)\,dx
&=\langle f,\mathcal{L}_i1_U\rangle
=\int_{F}f(x)\int_{\pi_i(F\setminus U)}L_i(\xi_i,\eta_i)\,d\eta_i\,dx
\\
&=\int_{\delta_i U}\int_U L_i(\xi_i,\eta_i)f(x)\,dx\,d\eta_i
\\
&=\int_{\delta_iU}\left[
\pi_{i,*}\mathcal{A}_if\right](\eta_i)\,d\eta_i
\end{align*}
as asserted.
\end{proof}

\section{Elliptic Divergence Operators}\label{EllipticDivergenceOp}

A suitable analogue of test functions on $U$ with compact support in the context of the operators $\mathcal{L}_i$ is given by
\[
\mathcal{D}_0(U)=\mathset{\phi\in\mathcal{D}(\closure_\delta U)\mid\forall\,i=1,\dots,d\,\forall\,x\in U\colon \phi(x)\int_{\delta_i U}L_i(\xi_i,\eta_i)\,d\eta_i=0}\,,
\]
where it is assumed that $x=(\xi_1,\dots,\xi_d)\in U$.

\subsection{Poisson equation}

A homogeneous second-order divergence operator is given as the following:
\begin{align}\label{divergenceHomogeneous}
P_2(\mathcal{L})u=\sum\limits_{i,j=1}^d\mathcal{L}_j\left(a^{ij}\mathcal{L}_iu\right)\,,
\end{align}
where $a^{ij}\colon F\to\mathds{R}$ are functions such that 
\[
a^{ij}=a^{ij}
\]
for $i,j=1,\dots,d$.
The operator $P_2(\mathcal{L})$ is called \emph{elliptic}, if the matrix $A=(a^{ij}(x))$ is positive definite for almost all $x\in F$, and the smallest positive eigenvalue of $A$ is always at least $\theta>0$.
Of interest is the boundary value problem of the Poisson equation:
\begin{align}\label{PoissonEquation}
\begin{cases}
\;P_2(\mathcal{L})u(x)\!\!\!\!&=f(x),\quad
x\in U
\\
\hfill u|_{\delta U}\!\!\!\!&=0
\end{cases}
\end{align}
with some given $f\in L^2(U)$.
A function $u\in W_0^{1,2}(U)$ is a \emph{weak solution} of (\ref{PoissonEquation}), if it holds true that
\[
\int_U\sum\limits_{i,j=1}^d\mathcal{L}_j(a^{ij}\mathcal{L}_iu)\phi\,dx=\int_Uf(x)\phi(x)\,dx
\]
for all $\phi\in W_0^{1,2}(U)$. 
From the self-adjointness of $\mathcal{L}_j$ (Theorem \ref{L2-Spectrum}), it follows that this is equivalent with
\begin{align*}
\int_U\mathcal{L}_j(a^{ij}\mathcal{L}_iu)\phi\,dx
&=\int_Ua^{ij}\mathcal{L}_iu\cdot\mathcal{L}_j\phi\,dx\,.
\end{align*}
Hence, 
\[
B_2[u,\phi]=\int_U\sum\limits_{i,j=1}^da^{ij}\mathcal{L}_iu\mathcal{L}_j\phi\,dx=\int_Uf(x)\phi(x)\,dx
\]
for all $\phi\in \mathcal{D}_0(U)$
is an equivalent formulation of $u$ being a weak solution of (\ref{PoissonEquation}).

\begin{Lemma}\label{alternativeEquation}
A function  $u$ is a weak solution of (\ref{PoissonEquation}), if and only if
\[
\int_U\sum\limits_{i,j=1}^d\mathcal{L}_{j,\phi}(a^{ij}\mathcal{L}_iu)\,dx=-\int_Uf(x)\phi(x)\,dx
\]
for all $\phi\in \mathcal{D}_0(U)$.
\end{Lemma}

\begin{proof}
It holds true that
\begin{align*}
\int_U\mathcal{L}_j(a^{ij}\mathcal{L}_iu\cdot\phi)\,dx
&=\langle a^{ij}\mathcal{L}_iu\cdot\phi,\mathcal{L}_j1_U\rangle
\\
&=\int_U a^{ij}(x)\mathcal{L}_iu(x)\,\phi(x)\int_{\pi_j(F\setminus U)}L_j(\xi_j,\eta_j)\,d\eta_j\,\,dx
\\
&=\int_Ua^{ij}(x)\mathcal{L}_iu(x)\,\phi(x)\int_{\delta_i(U)} L_j(\xi_j,\eta_j)\,d\eta_j\,dx =0\,,
\end{align*}
because 
\[
\phi(x)\int_{\delta_j U}L_j(\xi_j,\eta_j)\,d\eta_j=0
\]
for all $x\in U$.
From Lemmas \ref{LeibnizLike} and \ref{boundaryCondDist},
the assertion now follows.
\end{proof}

\begin{Ass}
It is assumed that $a^{ij}\in L^\infty(F)$ for $i,j=1,\dots,d$.
\end{Ass}

\begin{thm}[Energy estimates, homogeneous case]\label{energyHomogeneous}
There exist constants $\alpha,\beta>0$ such that
\begin{align*}
\absolute{B[u,v]}&\le
\alpha\norm{u}_{W^{1,2}_0(U)}
\norm{v}_{W^{1,2}_0(U)}
\\
\beta\norm{u}_{W^{1,2}_0(U)}
&\le B[u,u]
\end{align*}
for all $u,v\in W^{1,2}_0(U)$.
\end{thm}

\begin{proof}
First, observe that
\begin{align*}
\absolute{B[u,v]}\le\sum\limits_{i,j=1}^d\norm{a^{ij}}_{L^\infty}\int_U\absolute{\mathcal{L}_iu}\absolute{\mathcal{L}_jv}\,dx
\le\alpha\norm{u}_{W_0^{1,2}(U)}\norm{v}_{W_0^{1,2}(U)}
\end{align*}
for some $\alpha>0$. Next, ellipticity means that from the properties of the Rayleigh quotient, it follows that
\begin{align*}
\theta\int_U\sum\limits_{i=1}^d
\absolute{\mathcal{L}_iu}^2\,dx
\le\int_U
\sum\limits_{i,j=1}^d
a^{ij}\absolute{\mathcal{L}_iu}\absolute{\mathcal{L}_ju}
\,dx=B[u,u]
\end{align*}
Using the $p$-adic Poincar\'e inequality in  Proposition \ref{PoincareInequality},
the second inequality now follows.
\end{proof}

\begin{cor}\label{WeakSolutionDeg2}
Assume that the operator $P_2(\mathcal{L})$ is elliptic. Then (\ref{PoissonEquation}) has a unique weak solution in $W_0^{1,2}(U)$.
\end{cor}

\begin{proof}
This follows from the Lax-Milgram Theorem \cite[Theorem 6.2.1]{evans}.
\end{proof}

Now, let $P(\mathcal{L})$ be of the form
\begin{align}\label{OurOperator}
P(\mathcal{L})u=P_2(\mathcal{L})u+P_1(\mathcal{L})u+P_0(\mathcal{L})u
\end{align}
with $P_2(\mathcal{L})$ an operator as in 
(\ref{divergenceHomogeneous}), and
\begin{align*}
P_1(\mathcal{L})u&=\sum\limits_{i=1}^d b^i\mathcal{L}_iu,
\\
P_0(\mathcal{L})u&=cu\,,
\end{align*}
where again it is assumed that
\[
a^{ij},\;b^i,\;c\in L^\infty(\closure_\delta U)
\]
with $U\subseteq F$ open. The operator $P(\mathcal{L})$ is called \emph{elliptic}, if $P_2(\mathcal{L})$ is elliptic. Now, the boundary value problem is
\begin{align}\label{bvp}
\begin{cases}
\;
P(\mathcal{L})u(x)\!\!\!\!&=f(x),\quad x\in U
\\
\hfill u|_{\delta U}\!\!\!\!&=0
\end{cases}
\end{align}
with some given $f\in L^2(U)$.
A function $u\in W_0^{1,2}(U)$ is a weak solution of (\ref{bvp}), if
\[
\int_U\left(
\sum\limits_{i,j=1}^d\mathcal{L}_j(a^{ij}\mathcal{L}_iu)
+\sum\limits_{i=1}^db^i\mathcal{L}_iu+cu
\right)\phi\,dx=\int_Uf(x)\phi(x)\,dx
\]
for all $\phi\in\mathcal{D}_0(U)$. Again, it follows from the self-adjointness property (Theorem \ref{L2-Spectrum}) that
(\ref{bvp}) is equivalent with
\[
B[u,\phi]
=\int_U
\left(\sum\limits_{i,j=1}^d
a^{ij}\mathcal{L}_iu\mathcal{L}_j
+\sum\limits_{i=1}^d b^i\mathcal{L}_iu+cu\right)\phi\,dx=
\int_Uf(x)\phi(x)\,dx\,
\]
and, according to Lemma \ref{alternativeEquation}, the quadratic part can be replaced as
follows:
\[
B_2[u,\phi]=\int_U\sum\limits_{i,j=1}^da^{ij}\mathcal{L}_iu\mathcal{L}_j\phi\,dx=-\int_U\sum\limits_{i,j=1}^d\mathcal{L}_{j,\phi}(a^{ij}\mathcal{L}_iu)\,dx
\]
for $\phi\in\mathcal{D}_0(U)$.

\begin{thm}[Energy estimates]\label{energy}
There exist constants $\alpha,\beta>0$ and $\gamma\ge 0$ such that
\begin{align*}
\absolute{B[u,v]}&\le\alpha\norm{u}_{W_0^{1,2}(U)}
\norm{v}_{W_0^{1,2}(U)}
\\
\beta\norm{u}_{W_0^{1,2}(U)}^2
&\le B[u,u]+\gamma \norm{u}_{L^2(U)}^2
\end{align*}
for all $u,v\in W_0^{1,2}(U)$.
\end{thm}

\begin{proof}
The proof of the analogous classical result found in \cite[Theorem 2]{evans} can be adapted as follows:
\begin{align*}
\theta\int_U\absolute{\mathcal{L}u}^2\,dx&\le\int_U
\sum\limits_{i,j=1}^da^{ij}\mathcal{L}_iu\mathcal{L}_iu\,dx
\\
&=B[u,u]-\int_Ub^i\mathcal{L}_iu\,u+cu^2\,dx
\\
&\le B[u,u]+\sum\limits_{i=1}^d\norm{b^i}_{L^\infty}
\int_U\absolute{\mathcal{L}u}\absolute{u}\,dx
+\norm{c}_{L^\infty}\int_U u^2\,dx\,.
\end{align*}
Using Cauchy's inequality with $\epsilon$, cf.\ \cite[\S B.2]{evans}, obtain
\[
\int_U\absolute{\mathcal{L}u}\absolute{u}\,dx\le\epsilon\int_U\absolute{\mathcal{L}u}^2\,dx+\frac{1}{4\epsilon}\int_Uu^2\,dx
\]
for $\epsilon>0$. Choose $\epsilon>0$ so small that
\[
\epsilon\sum\limits_{i=1}^d\norm{b^i}_{L^\infty}<\frac{\theta}{2}\,.
\]
This implies
\[
\frac{\theta}{2}\int_U\absolute{\mathcal{L}u}^2\,dx\le B[u,u]+C\int_Uu^2\,dx
\]
for some $C>0$. Again, using the $p$-adic Poincar\'e inequality
in  Proposition \ref{PoincareInequality},
obtain
\[
\beta\norm{u}_{W_0^{1,2}(U)}^2\le B[u,u]+\gamma\norm{u}_{L^2(U)}^2
\]
for suitable $\beta>0$, $\gamma\ge0$. This proves the assertions.
\end{proof}

\begin{cor}\label{weak_bvp_perturbed}
There is a number $\gamma\ge0$ such that for all $\mu\ge\gamma$ and every $f\in L^2(U)$, there exists a weak solution in $W_0^{1,2}(U)$ of the boundary value problem
\begin{align}\label{bvp_perturbed}
\begin{cases}
\;P(\mathcal{L})u(x)+\mu u(x)\!\!\!\!&=f(x),\quad x\in U
\\
\hfill u|_{\delta U}\!\!\!\!&=0
\end{cases}
\end{align}
for $U\subseteq F$ open.
\end{cor}

\begin{proof}
Due to the energy estimates of Theorem \ref{energy}, the proof of \cite[Theorem 6.2.3]{evans} carries over.
\end{proof}

\subsection{Diagonalisability of elliptic divergence operators}

Let 
\[
\phi,\tilde\phi\in\mathcal{E}:=
\mathset{\text{eigenbasis for $\mathcal{L}_1$}}\otimes\dots\otimes\mathset{\text{eigenbasis for $\mathcal{L}_d$}}
\]
where ``eigenbasis for $\mathcal{L}_i$'' refers to the orthonormal basis of $L^2(\pi_i(U))$ consisting of eigenfunctions of the push-forward component operator $\pi_{i,*}\mathcal{L}_i$, cf.\ Theorem \ref{ComponentSpectrum}.

\begin{Proposition}\label{P2Diagonalisable}
The operator  $P_2(\mathcal{L})$ is 
densely defined on $L^2(U)$. It is also
self-adjoint on $L^2(U)$. If further
$a^{ij}\in \mathcal{D}(U)$ for all $i,j=1,\dots,d$, then the spectrum of $P^2(\mathcal{L})$ on $L^2(U)$ is a point spectrum. If, furthermore, $P_2(\mathcal{L})$ is elliptic, then 
$P_2(\mathcal{L})$ 
is positive semi-definite.
\end{Proposition}

\begin{proof}
The operator 
\[
P_2(\mathcal{L})=
\sum\limits_{i,j=1}^d\mathcal{L}_iM_{a^{ij}}\mathcal{L}_j\,,
\]
is densely defined on $L^2(U)$, because each summand is defined on the space of test functions $\mathcal{D}(U)$.

\smallskip
Next, observe that
\[
P_2(\mathcal{L})^*
=\sum\limits_{i,j=1}^d
\mathcal{L}_jM_{a^{ij}}\mathcal{L}_i
=P_2(\mathcal{L})
\]
by the double summation. Hence, $P_2(\mathcal{L})$ is self-adjoint.

\smallskip
In order to see the point spectrum, observe  that for $\phi\in\mathcal{E}$, there is an $\mathcal{E}$-expansion:
\begin{align*}
P_2(\mathcal{L})\phi&=
\sum\limits_{\phi',\phi''}
\sum\limits_{i=1}^d
\lambda_{\phi,j}\lambda_{\phi'',i}
\langle a^{ij},\phi'\rangle\langle\phi\phi',\phi''\rangle\phi''\,.
\end{align*}
Now, observe that
\[
\sum\limits_{\phi'}\langle a^{ij},\phi'\rangle
\langle\phi\phi',\phi''\rangle
=\sum\limits_{\phi'}\langle a^{ij},\phi'\rangle\langle\phi',\bar\phi\phi''\rangle
=\langle a^{ij},\bar\phi\phi''\rangle\,,
\]
because the middle sum is an $\ell^2$-inner product, and this coincides with the $L^2$-inner product on the left hand side. Hence,
\begin{align*}
P_2(\mathcal{L})\phi
&=\sum\limits_{\phi''}\left\langle
\sum\limits_{i,j=1}^da^{ij}\lambda_{\phi,j}\lambda_{\phi'',i},\bar\phi\phi''\right\rangle\phi''
\\
&=\sum\limits_{\phi''}\left\langle \phi
\sum\limits_{i,j=1}^da^{ij}\lambda_{\phi,j}\lambda_{\phi'',i},\phi''
\right\rangle\phi''\,.
\end{align*}
Since the $a^{ij}\in\mathcal{D}(U)$ for $i,j=1,\dots,d$,
it follows that for fixed $\phi\in\mathcal{E}$, the last sum is a finite linear combination of $\phi''\in\mathcal{E}$.
By self-adjointness of the $M_{a^{ij}}$, it follows also that given $\phi''$, there are only finitely many $\phi\in\mathcal{E}$ such that
the matrix
\[
P_2(\phi,\phi'')
=\left\langle
\phi\sum\limits_{i,j=1}^da^{ij}\lambda_{\phi,j}\lambda_{\phi'',i},\phi''
\right\rangle
\]
has only finitely many $\phi''\in\mathcal{E}$ with non-zero values.
Since $\mathcal{E}$ is an orthonormal basis of $L^2(U)$, it now follows that this space decomposes into $P_2(\mathcal{L})$-invariant finite-dimensional subspaces which will be denoted as $V_\phi$ for $\phi\in\mathcal{E}$. Notice that for $\phi'\neq \phi$, 
it may happen that $V_{\phi'}=V_\phi$. The restriction of $P_2(\mathcal{L})$ to $V_\phi$ is the left-multiplication by a matrix of the form
\[
\sum\limits_{i,j=1}^d C_{\phi,ij}
\]
with
\[
C_{\phi,ij}=D_{\phi,i}A_{\phi,ij} D_{\phi,j}\,,
\]
where $A_{\phi,ij}$ is symmetric, and $D_{\phi,i}$, $D_{\phi,j}$ are diagonal matrices for $i,j=1,\dots,d$.
It follows that
\[
\left(\sum\limits_{i,j=1}^d
C_{\phi,ij}\right)^\top
=\sum\limits_{i,j=1}^d
D_{\phi,j}A_{\phi,ij}D_{\phi,i}
\stackrel{(*)}{=}\sum\limits_{i,j=1}D_{\phi,j}A_{\phi,ji}D_{\phi,i}
=\sum\limits_{i,j=1}^dC_{\phi,ij}\,,
\]
where $(*)$ holds true because $a^{ij}=a^{ji}$ for $i,j=1,\dots,d$. Hence, this matrix is symmetric.
This implies that $P_2(\mathcal{L})$ is unitarily diagonalisable as an operator on $L^2(U)$, and thus its spectrum is a point spectrum, as asserted.

\smallskip
The non-negativity of the eigenvalues follows thus: let $\phi\in\mathcal{E}$. From
\[
P_2(\mathcal{L})\phi
=\sum\limits_{\phi''}P_2(\phi,\phi'')\phi''
\]
it follows that
\begin{align*}
\langle P_2(\mathcal{L})\phi,\phi\rangle
&=
P_2(\phi,\phi)
=\left\langle\phi
\sum\limits_{i,j=1}^da^{ij}\lambda_{\phi,j}\lambda_{\phi,i},\phi
\right\rangle
\\
&=\lambda_{\phi,j}\lambda_{\phi,j}
\sum\limits\left\langle
\phi\sum\limits_{i,j=1}^d a^{ij},\phi
\right\rangle
\ge 0\,,
\end{align*}
because the matrix $(a^{ij})$ is positive semi-definite almost everywhere on $U$ (actually everywhere on $U$ by assumption).
Hence, $P_2(\mathcal{L})$ is positive semi-definite, since for $f\in L^2(U)$, it now follows that
\begin{align*}
\langle P_2(\mathcal{L})f,f\rangle&=
\sum\limits_{\phi}\langle f,\phi\rangle\langle\phi'',f\rangle
\sum\limits_{\phi''}P_2(\phi,\phi'')\langle\phi'',\phi\rangle
\\
&=\sum\limits_{\phi}\absolute{\langle f,\phi\rangle}^2
P_2(\phi,\phi)\ge0
\end{align*}
as asserted.
\end{proof}

In order to address the homogeneous part of degree one, observe first that
\begin{align}\label{P1adjoint}
P_1(\mathcal{L})^*v=\sum\limits_{i=1}^d\mathcal{L}_i(b^iv)
\end{align}
for $v\in L^2(U)$. It is assumed that $b^i\in \mathcal{D}(U)$ for $i=1,\dots,d$.

\begin{Proposition}[Normality Condition]\label{normalityCondition}
The operator $P_1(\mathcal{L})$ is normal, if and only if
\begin{align}\label{normalityEq}
&\sum\limits_{i,j=1}^d
\lambda_{\phi,i}\lambda_{\psi,j}
\sum\limits_{\phi'}
\langle\phi,b^i\phi'\rangle
\langle b^j\phi',\psi\rangle
=
\sum\limits_{i,j=1}^d
\sum\limits_{\phi'}
\lambda_{\phi',i}\lambda_{\phi',j}
\langle\phi,b^i\phi'\rangle
\langle b^j\phi',\phi'\rangle
\end{align}
for all $\phi,\psi\in\mathcal{E}$.
\end{Proposition}

\begin{proof}
First, observe the expansions, using also (\ref{P1adjoint}):
\begin{align}\label{P1expansion}
P_1(\mathcal{L})\phi&=
\sum\limits_{i=1}^d 
\sum\limits_{\phi'}
\lambda_{\phi,i}
\langle \phi\, b^i,\phi'\rangle\phi'
\\\nonumber
P_1(\mathcal{L})^*\phi
&=\sum\limits_{i=1}^d
\sum\limits_{\phi'}
\lambda_{\phi',i}
\langle\phi\, b^i,\phi'\rangle\phi'\,,
\end{align}
from which it follows that
\begin{align}\label{normalityCond1}
\langle P_1(\mathcal{L})\phi,P_1(\mathcal{L})\psi\rangle
&=\sum\limits_{i,j=1}^d
\lambda_{\phi,i}\lambda_{\psi,j}
\sum\limits_{\phi'}\langle\phi,b^i\phi'\rangle
\langle b^j\phi',\psi\rangle
\\\label{normalityCond2}
\langle P_1(\mathcal{L})^*\phi,P_1(\mathcal{L})^*\psi\rangle
&=
\sum\limits_{i,j=1}^d
\sum\limits_{\phi'}
\lambda_{\phi',i}\lambda_{\phi',j}
\langle\phi,b^i\phi'\rangle
\langle b^j\phi',\psi\rangle\,.
\end{align}
Observe that the domain of each of the operators $P_1(\mathcal{L})$ and $P_1(\mathcal{L})^*$ is the intersection of the domains of the operators $\mathcal{L}_i$. Thus, they are both the same.
Hence, the normality of $P_1(\mathcal{L})$ is equivalent to the equality of 
expressions (\ref{normalityCond1}) and (\ref{normalityCond2}), as asserted.
\end{proof}

\begin{remark}
Assume that $b^1,\dots,b^d\in \mathcal{D}(U)$.
Then  Proposition \ref{normalityCondition} says that the eigenvalues 
$\lambda_{\phi,i}$ for $\phi\in\mathcal{E}$, $i=1,\dots,d$ satisfy an algebraic condition  given by Proposition \ref{normalityCondition}, if and only if $P_1(\mathcal{L})$ is normal. The reason why this condition is algebraic, is because the functions $b^1,\dots,b^d$ are test functions.
\end{remark}

\begin{Lemma}\label{P1diagonalisable}
The operator $P_1(\mathcal{L})$ is diagonalisable on $L^2(U)$, and its spectrum is a point spectrum, if each eigenspace of $\mathcal{L}_i$ is invariant under each $M_{b^i}$, where $i=1,\dots,d$. 
\end{Lemma}

\begin{proof}
Observe that the restrictions $C_{\phi,i}$ of
the operators $M_{b^i}\mathcal{L}_i$ onto the invariant finite-dimensional subspaces $V_\phi$ given by
\[
P_1(\mathcal{L})\phi
=\sum\limits_{i=1}^d\lambda_{\phi,i}\sum\limits_{\phi'}\langle\phi b^i,\phi'\rangle\phi'
\]
are all diagonalisable, because
\[
C_{\phi,i}=D_{\phi,i}B_{\phi,i}
\]
with $B_{\phi,i}$ a symmetric and $D_{\phi,i}$ a diagonal matrix, and
\[
C_{\phi,i}D_{\phi,i}=D_{\phi,i}B_{\phi,i}D_{\phi,i}=D_{\phi,i}C_{\phi,i}^\top\,,
\]
which is known as the \emph{detailed balance property} \cite[(4.1) and (6.15)]{Kampen}.

\smallskip
Since $V_\phi$ is a direct sum of eigenspaces of $P_1(\mathcal{L})$, it follows
under the hypothesis, that the restrictions of the operators $M_{b^i}\mathcal{L}_i$ to $V_\phi$  mutually commute. Hence, since these operators are themselves diagonalisable, they are also
 simultaneously diagonalisable. This implies the diagonalisablity of $P_1(\mathcal{L})$ and the point spectrum property, as asserted.
 \end{proof}

\begin{thm}\label{Pdiagonalisable}
Let $P(\mathcal{L})$ be the operator (\ref{OurOperator}) with $a^{i,j},b^i,c\in \mathcal{D}(U)$ for $i,j=1,\dots,d$ on $L^2(U)$ for $U\subset F$ an open domain.
Assume further that the eigenspaces of $\mathcal{L}$ are invariant under $M_{b^i}$ for $i=1,\dots,d$, or that $P_1(\mathcal{L})$ is normal.
Moreover, assume that
\[
P_k(\mathcal{L})P_\ell(\mathcal{L})=P_\ell(\mathcal{L})P_k(\mathcal{L})
\]
holds true for $k,\ell=0,1,2$.
Then $P(\mathcal{L})$ is unitarily diagonalisable, its  spectrum is a point spectrum, and all eigenvalues have only finite multiplicity. 
\end{thm}

\begin{proof}
According to Proposition \ref{P2Diagonalisable}, $P_2(\mathcal{L})$ is unitarily diagonalisable, and according to Lemma \ref{P1diagonalisable}, $P_1(\mathcal{L})$ is diagonalisable, or in the other case, $P_1(\mathcal{L})$ is unitarily diagonalisable. Also, $P_0(\mathcal{L})$ is unitarily diagonalisable.
Since these operators mutually commute, they are simultaneously diagonalisable, and also unitarily diagonalisable. 
This now proves that $P(\mathcal{L})$ is unitarily diagonalisable.

\smallskip
Similarly as in the proof of Proposition \ref{P2Diagonalisable}, arrive at the equality
\begin{align}\label{Pphi}
P(\mathcal{L})\phi
=\sum\limits_{\phi'}
\left\langle
\phi\left[\sum\limits_{i=1}^d
\left(\sum\limits_{j=1}^d
\lambda_{\phi,i}a^{ij}\lambda_{\phi',j}\right)+\lambda_{\phi,i}b^i
+c
\right],\phi'
\right\rangle\phi'\,,
\end{align}
which gives, similarly as in the proof of Proposition \ref{P2Diagonalisable}, the restriction of $P(\mathcal{L})$ to a finite-dimensional subspace $V_\phi$ of $L^2(U)$ as left-multiplication with the matrix
\[
W_\phi=\sum\limits_{i=1}^d
\sum\limits_{j=1}^d
C_{\phi,ij}+C_{\phi,i}+C_\phi\,,
\]
where
\begin{align}\label{matrices}
C_{\phi,ij}=D_{\phi,i}A_{\phi,ij}D_{\phi,j},\;
C_{\phi,i}=D_{\phi,i}B_{\phi,i}
\end{align}
with $D_{\phi,i}$, $D_{\phi,j}$ diagonal and $A_{\phi,ij},B_{\phi,i},C$ symmetric. 
The multiplicity of the eigenvalues of $P(\mathcal{L})$ is determined by the multiplicity of the eigenvalues $\lambda_{\phi,i}$ of $\pi_{i,*}\mathcal{L}_i$, which are finite for $i=1,\dots, d$, cf.\ Theorem \ref{ComponentSpectrum}.
Using (\ref{Pphi}) and the block structure given by the invariant subspaces $V_\phi$, this implies the finite multiplicity  of the eigenvalues of $P(\mathcal{L})$,  and also that its spectrum is a point spectrum.
\end{proof}

\begin{remark}
In the context of random walks, the condition $C=DB$ with $B$ symmetric and $D$ a diagonal matrix is called
\emph{detailed balance property}, because 
\[
D^{-1}C=B=B^\top=C^\top D^{-1}
\;\Leftrightarrow\;
CD=DC^\top
\]
corresponds to \cite[eq.\ (4.2)]{Kampen} with $D$ playing the role of a stationary distribution.
It will here also be said that $P_1(\mathcal{L})$ satisfies the detailed balance condition, because of (\ref{matrices}).
Similarly, we can also say that $P_2(\mathcal{L})$ satisfies a generalised form of the detailed balance condition, so that the operator $P(\mathcal{L})$ can be viewed as belonging to a kind of a balanced process.
\end{remark}

The space of $L^q$-functions on $\closure_\delta U$ satisfying the boundary condition $u|_\delta U=0$ will be denoted as $L^q_0(U)$ for $0<q\le\infty$.

\begin{cor}\label{eigenTestFunctions}
Under the hypotheses of Theorem \ref{Pdiagonalisable}, the eigenfunctions of $P(\mathcal{L})$ are in $\mathcal{D}(U)$. In this case,  $L_0^2(U)$ is invariant under $P(\mathcal{L})$ acting as an unbounded operator,  and the restricted operator is also unitarily diagonalisable with  point spectrum, and with eigenfunctions in $\mathcal{D}_0(U)$. 
\end{cor}

\begin{proof}
From the representation (\ref{Pphi}), i.e.\ the $P(\mathcal{L})$-invariant finite-dimensional subspaces $V_\phi\subset L^2(U)$, it follows that the eigenfunctions of $P(\mathcal{L})$ are linear combinations of 
the functions $\phi\in\mathcal{E}$, i.e.\ they belong to $\mathcal{D}(U)$. 

\smallskip
The space $\mathcal{D}_0(U)$ certainly contains the polywavelets supported in $U$, and the graph test functions $f$ w.r.t.\ the coordinate subgraphs conditioned by $f|_{\delta U}=0$. Notice that the vertices in this case do not always correspond to $p$-adic discs, similarly as in the case of Mumford curves \cite[Section 4.1]{brad_HeatMumf}.
The latter graphs define in each coordinate a finite-dimensional subspace invariant under $\mathcal{L}_i$, cf.\ Remark \ref{subsetting}. It follows that the proof of Theorem \ref{Pdiagonalisable} carries over to $P(\mathcal{L})$ restricted to
$L_0^{2}$, which is thus seen to be invariant in the unbounded sense, and with  eigenfunctions being test functions, hence in $\mathcal{D}_0(U)$, also in this case. This proves the assertions.
\end{proof}

\section{Heat kernels and Green function}

Here, the following is assumed:

\begin{Ass}\label{assumption}
It is assumed that $P(\mathcal{L})$ is elliptic, satisfies the hypothesis of Theorem \ref{Pdiagonalisable}, and that its eigenvalues are non-negative. 
\end{Ass}

The action of $e^{-tP(\mathcal{L})}$ on $W^{k,2}_0(U)$ for $k\in\mathds{N}$ is of interest in the study of diffusion under boundary conditions. These are Hilbert spaces according to Corollary \ref{SobolevHilbertSpace} and Proposition \ref{SobolevClosedSubspace}.

\begin{Lemma}
The semigroup $e^{-tP(\mathcal{L})}$ acts compactly on $W_0^{k,2}(U)$ for $t>0$ and $k\in\mathds{N}$.
\end{Lemma}

\begin{proof}
The operators $e^{-tP(\mathcal{L})}$ for $t>0$ are trace-class as operators on the Hilbert spaces $W_0^{k,2}$ by Assumption \ref{assumption}.
\end{proof}

Let $x_0\in U$. The Green function for the diffusion equation
\begin{align}\label{PDiffusion}
\frac{\partial}{\partial t}u(x,t)+P(\mathcal{L})u(x,t)=0
\end{align}
on $U$ under the boundary condition $u(\cdot,0)|_{\delta U}=0$
is given as a solution of the following Poisson equation:
\begin{align}\label{PoissonGreen}
\begin{cases}\;P(\mathcal{L})G(x,x_0)\!\!\!\!&=\delta(x-x_0),\quad x\in U
\\
\hfill G(x,x_0)\!\!\!\!&=0,\quad\quad\qquad\;\;x\in\delta U\,,
\end{cases}
\end{align}
where the constants $\mu,\gamma$ as in (\ref{bvp_perturbed}) are not going to be required, because of the stronger assumptions on $P(\mathcal{L})$.
The Green function  is related to the heat kernel via
\begin{align}\label{GreenFunction}
G(x,y)=\int_0^\infty h(x,y,t)\,dt
\end{align}
with
\begin{align*}
h(x,y,t)=\sum\limits_{\psi\atop\lambda_\psi>0}
e^{-\lambda_\psi t}\psi(x)\overline{\psi(y)}
\end{align*}
being the part of the heat kernel
\begin{align}\label{heatKernel}
H(x,y,t)=h(x,y,t)+\sum\limits_{\psi\atop\lambda_\psi=0}\psi(x)\overline{\psi(y)}
\end{align}
associated with (\ref{PDiffusion}), and where $\psi$ runs through an eigenbasis of $W^{k,2}_0(U)$ for $P(\mathcal{L})$. According to Corollary \ref{eigenTestFunctions}, these exist and are test functions.
The function $H(x,y,t)$, if convergent, is formally the heat kernel for $P(\mathcal{L})$-diffusion under boundary conditions with $U\subseteq F$ open.
\newline

In order to prove the existence 
of the Green function, the strategy will be to prove the convergence of 
$H(x,y,t)$ for $t\ge0$, as well as of the right hand side of (\ref{GreenFunction}) 
in the generality of (\ref{PoissonGreen}).

\subsection{Markov property}
In order to rightly say that $P(\mathcal{L})$ defines a diffusion,  the Markovian semigroup property is established first under Assumption \ref{assumption}.

\begin{thm}\label{MarkovSemigroup}
The operator $-P(\mathcal{L})$ generates a contraction semigroup $e^{-tP(\mathcal{L})}$ with $t\ge0$ on $W_0^{k,2}(U)$ for $k\in\mathds{N}$, and the action satisfies the Markov property if $k\ge2$.
\end{thm}

\begin{proof}
Since the operator $-P(\mathcal{L})$ acts on the Hilbert space $L_0^{2}(U)$ (cf.\ Corollary \ref{eigenTestFunctions}),  and its eigenvalues are bounded from above (they are non-positive by Assumption \ref{assumption}), it follows that $e^{-tP(\mathcal{L})}$ is a strongly continuous semigroup acting on $L_0^2(U)$ for $t\ge0$.
Due to the non-positiveness of the eigenvalues of $-P(\mathcal{L})$, the spaces $W_0^{k,2}(U)$ are invariant under $e^{-tP(\mathcal{L})}$ for $t\ge0$, and the semigroup is also strongly continuous on these Hilbert spaces.

\smallskip
The semigroup $e^{-P(\mathcal{L})}$ with $t\ge0$ is also a contraction semigroup, because
\begin{align}\label{contraction}
\norm{\int_0^t e^{-\tau P(\mathcal{L})}u\,d\tau}_{W_0^{1,2}(U)}\le t\norm{u}_{W_0^{1,2}(U)}\,,
\end{align}
which can readily be seen for eigenfunctions first, and then for linear combinations of such using Pythagoras. The reason, why (\ref{contraction}) implies the semigroup to be contractive is that with
\[
R(\lambda)u=\lambda \int_0^\infty e^{-\lambda t}\int_0^t e^{-\tau P(\mathcal{L})}u\,d\tau\,dt
\]
being an expression for the resolvent:
\[
R(\lambda)=\left(\lambda+P(\mathcal{L})\right)^{-1}\,,
\]
it follows that
\begin{align*}
\norm{R(\lambda)u}_{W_0^{1,2}(U)}
&\le \lambda \int_0^\infty\norm{\int_0^t e^{-\tau P(\mathcal{L})}u\,d\tau}_{W_0^{1,2}(U)}\,dt
\\
&\le \lambda\int_0^\infty e^{-\lambda t}t\norm{u}_{W_0^{1,2}(U)}\,dt
\\
&=\frac{1}{\lambda}
\norm{u}_{W_0^{1,2}(U)}\,,
\end{align*}
implying that
\[
\norm{\lambda+P(\mathcal{L})}^{-1}\le \frac{1}{\lambda}\,,
\]
and thus, using the Hille-Yosida Theorem for contraction semigroups \cite[Theorem II.3.5]{EN2006}, it follows that $e^{-tP(\mathcal{L})}$ with $t\ge0$ is a contraction semigroup on $W_0^{k,2}(U)$ with $k\in\mathds{N}$. 

\smallskip
The Markovian property for $k\ge2$ follows from first showing that 
\begin{align}\label{mark1}
f\ge 0\;\text{a.e}\;&\Rightarrow\;e^{-tP(\mathcal{L})}f\ge0\;\text{a.e.}
\\\label{mark2}
f\le 1\;\text{a.e.}\;&\Rightarrow\;e^{-tP(\mathcal{L})}f\le 1\;\text{a.e.}\,.
\\\label{massPreserving}
e^{-tP(\mathcal{L})}1_U&=1_U\,,
\end{align}
and then by exhibiting an invariant measure for $e^{-tP(\mathcal{L})}$ with $t\ge0$.

\smallskip
Statement (\ref{mark1}) is seen thus: $f\ge0$ means that it is a linear combination of eigenfunctions which is invariant under the action of $\left(\mathds{F}_p^\times\right)^d$ via $x\mapsto \underline{j}x=(j_1\xi_1,\dots,j_d\xi_d)$, where
$x=(\xi_1,\dots,\xi_d)\in \closure_\delta U$, and $\underline{j}=(j_1,\dots,j_d)\in\left(\mathds{F}_p^\times\right)^d$. For here, the group $\left(\mathds{F}_p^\times\right)^d$ is called \emph{torus}.
It follows that $f$ is a positive linear combination of torus-invariant sums of eigenfunctions. By the invariance of the eigenspaces under the torus action, cf.\ (\ref{Pphi}) and the $\mathds{F}_p^\times$-invariance of the coordinate Laplacians $\mathcal{L}_i$ for $i=1,\dots,d$ (cf.\ Theorem \ref{ComponentSpectrum}), it follows that this is also the case for $e^{-tP(\mathcal{L})}f$. Property (\ref{mark2}) is verified in a similar manner, because all eigenvalues of $-P(\mathcal{L})$ are non-positive. Property (\ref{massPreserving}) follows from the fact that $1_U$ is an eigenfunction of $-P(\mathcal{L})$ with eigenvalue $0$. 

\smallskip
In order to find an invariant measure, 
the detailed balance condition (\ref{matrices}) can be used by taking for each of the finite-dimensional invariant subspaces $V_\phi$ an the invariant measure $\pi_\phi$ for the semigroup 
\[
e^{-tP(\mathcal{L})_\phi}:=e^{-tP(\mathcal{L})}|_{V_\phi}\,.
\]
It satisfies
\[
e^{-tP(\mathcal{L})_\phi}\pi_\phi f_\phi=\int_U f_\phi(y)\,d\pi_\phi(y)
\]
for $f_\phi\in V_\phi$. Write $f\in\mathcal{D}_0(U)$ as a (finite) sum
\[
f=\sum\limits_{V_\phi}f_\phi\,,
\]
where $f_\phi$ is the orthogonal projection of $f$ onto $V_\phi$. Then, by taking formally
\[
\pi=\sum\limits_{V_\phi}\pi_\phi
\]
as a direct sum, observe that
\[
e^{-tP(\mathcal{L})}\pi f
=\sum\limits_{V_\phi}e^{-tP(\mathcal{L})_\phi}\pi_\phi f_\phi
=\sum\limits_{V_\phi}\int_U f_\phi(y)\,d\pi_\phi(y)
=\int_Uf(y)\,d\pi(y)\,,
\]
i.e.\ $\pi$ is a distribution on $\mathcal{D}_0(U)$. In order to see that it is also one on $W_0^{k,2}(U)$ for $k\ge2$,  approximate $f\in W_0^{k,2}(U)$ with a convergent sequence of test functions $f^{(n)}\in\mathcal{D}_0(U)$, and observe that
\[
\sum\limits_{V_\phi}e^{-tP(\mathcal{L})_\phi}\pi_\phi f_\phi^{(n)}=
\sum\limits_{V_\phi}\int_Uf^{(n)}_\phi(y)\,d\pi_\phi(y)
=e^{-tP(\mathcal{L})}\pi f^{(n)}
\]
converges for $n\to\infty$ to
\begin{align}\label{expression}
\int_U f\,d\pi=\sum\limits_{V_\phi}\int_Uf_\phi\,d\pi_\phi=\sum\limits_{V_\phi}e^{-tP(\mathcal{L})_\phi}\pi_\phi f_\phi = \left(\sum\limits_{V_\phi}e^{-tP(\mathcal{L})_\phi}\pi_\phi\right)f\,,
\end{align}
and this does converge for the following reason: first, observe from (\ref{Pphi}) and (\ref{matrices}) that $\pi_\phi\in V_\phi$ is a tuple containing 
expressions of the form 
\[
\epsilon_2\lambda_{\phi,i}\lambda_{\phi,j}+\epsilon_1\lambda_{\phi,i}+\epsilon_0
\]
with $\lambda_{\phi,\ell}$ the eigenvalue of $P(\mathcal{L})$ corresponding to eigenfunction $\phi$, $\epsilon_r\in\mathset{0,1}$
for $r=0,1,2$ and $i=1,\dots,d$, in their respective order and multiplicities. So, for $f\in W_0^{k,2}(U)$ with $k\ge2$, it holds true that
\begin{align*}
\absolute{\int_Uf\,d\pi}
&=\absolute{\langle f,\pi\rangle}\le\sum\limits_{V_\phi}\absolute{\langle f_\phi,\pi_\phi\rangle}
\\
&\le
\sum\limits_{\phi}
\left(\epsilon_0+\epsilon_1\sum\limits_{i=1}^d\lambda_{\phi,i}+\epsilon_2\sum\limits_{i,j=1}^d
\lambda_{\phi,i}\lambda_{\phi,j}
\right)\absolute{f_\phi}
\\
&\le\norm{f^{\frac12}}_{W_0^{2,2}(U)}^2<\infty\,,
\end{align*}
where
\[
f=\sum\limits_{\phi}f_\phi\phi,\quad f_\phi\in\mathds{C},
\]
is the orthogonal eigendecomposition in $W_0^{k,2}(U)$,
because if $f\in W_0^{k,2}(U)$, then $f^{\frac12}\in W_0^{k,2}(U)$ due to the exponential growth of the eigenvalues of the component Laplacians.
This means that 
\[
\sum\limits_{V_\phi}e^{-tP(\mathcal{L})_\phi}\pi_\phi\in W_0^{k,2}(U)'
\]
is a distribution on $W_0^{k,2}(U)$ for $k\ge2$, which coincides with the formally given distribution
\[
e^{-tP(\mathcal{L})}\pi
\]
together with the identity (\ref{expression}).
Hence, $\pi$ is the distribution on $W_0^{k,2}(U)$ for $k\ge2$, invariant under $e^{-tP(\mathcal{L})}$ for $t>0$.
This now proves the assertions.
\end{proof}

\begin{cor}\label{kernelRepresentation}
The semigroup $e^{-tP(\mathcal{L})}$ with $t\ge0$ has a kernel representation $p_t(x,\cdot)$ for $t\ge0$, $x\in\closure_\delta U$, i.e.\ the map $A\mapsto p_t(x,A)$ is a Borel measure and it holds true that
\[
\int_U p_t(x,dy)f(y)=e^{-tP(\mathcal{L})}f(x)
\]
for $f\in W_0^{k,2}(U)$ with $k\ge2$.
\end{cor}

\begin{proof}
The operator $e^{-tP(\mathcal{L})}$ takes positive measurable functions to positive measurable functions, cf.\ (\ref{mark1}). Further, 
it holds true that $e^{-tP(\mathcal{L})}1_{\closure_\delta U}=1_{\closure_\delta U}$. Also,  $e^{-tP(\mathcal{L})}$ is bounded on $L^1_0(U)$ for $t\ge0$, because the eigenvalues of $P(\mathcal{L})$ are non-negative. Thus, according to \cite[Proposition 1.2.3]{BGL2014}, a kernel representation $p_t(x,\cdot)$ exists for $x\in\closure_\delta U$, and $t\ge0$. 
\end{proof}

\subsection{Convergence of heat kernel and Green function}

The goal here is to show that $p_t(x,\cdot)$ given by Corollary \ref{kernelRepresentation}
has a probability density function given by the function 
$H(x,y,t)$ in  (\ref{heatKernel}),
 and to show the convergence of the corresponding expression of the associated Green function.

\begin{thm}\label{heatKernelFunction4Semigroup}
The Markov semigroup $e^{-tP(\mathcal{L})}$ on $W_0^{k,2}(U)$ has a heat kernel function given by $H(x,y,t)\in L^2_0(\closure_\delta U)\otimes L^2_0(\closure_\delta U)$ for $t>0$.
\end{thm}

\begin{proof}
In light of Theorem \ref{MarkovSemigroup} and Corollary \ref{kernelRepresentation}, it suffices to prove that $H(x,y,t)\in L^2_0(\closure_\delta U)\otimes L^2_0(\closure_\delta U)$ for $t>0$. 

\smallskip
Assume $x=y$, $t>0$. Then $H(x,x,t)$ for $x\in\closure_\delta U$ is the trace of $e^{-P(\mathcal{L})}$ which is finite by Assumption \ref{assumption}.

\smallskip
Assume $x\neq y$, $t>0$. Since
\[
\absolute{\psi(x)\overline{\psi(y)}}\le\mu(\closure_\delta U)\,
\]
it follows that
\[
\absolute{H(x,y,t)}\le\sum\limits_\psi e^{-t\lambda_\psi}<\infty
\]
for $t>0$ by Assumption \ref{assumption}. This proves the assertion.
\end{proof}

\begin{cor}\label{GreenFunctionExists}
The Green function $G(x,y)$ associated with $-P(\mathcal{L})$ exists and is given by
\[
G(x,y)=\sum\limits_{\psi\atop\lambda_\psi>0}
\lambda_\psi^{-1}\psi(x)\overline{\psi(y)}
\]
for $x,y\in\closure_\delta U$.
\end{cor}

\begin{proof}
The expression (\ref{GreenFunction}) yields the asserted sum. Its convergence follows from the unbounded growth of the eigenvalues as follows:
Theorem \ref{ComponentSpectrum} says that for $\phi\in\mathcal{E}$, it holds true that
\[
\lambda_{\phi,i}\in O\left(p^{n(1+\alpha_i)}\right)
\]
for $\mu_i(\supp(\phi))=
p^{-n}$ with $n>>0$.
Let 
\[
\alpha=\max\mathset{\alpha_1,\dots,\alpha_d}\,.
\]
Then with (\ref{Pphi}) it follows that
\[
\lambda_\psi\in O\left(p^{2dn(1+\alpha)}\right)\,,
\]
where $\psi$ is assumed to be finite sum of  eigenfunctions $\phi\in\mathcal{E}$ having support maximally of volume $p^{-dn}$ for $n>>0$. Since the value
\[
\absolute{\psi(x)\overline{\psi(y)}}\le\mu(\closure_\delta U)
\]
is bounded for $x,y\in\closure_\delta U$, the asserted convergence now follows.
\end{proof}

\section*{Acknowledgements}
The author wants to thank 
Martin Breunig, Bastian Erdn\"u{\ss}, Markus Jahn, \'Angel Mor\'an Ledezma, David Weisbart, 
and
Wilson Z\'{u}\~{n}iga-Galindo
as well as Andew and John Bradley, 
for fruitful discussions.
User Mizar is thanked for helpful answers to some relevant posts on StackExchange.
This research is partially funded by the Deutsche For\-schungsgemeinschaft under 
project number 469999674.

\bibliographystyle{plain}
\bibliography{biblio}

\begin{thebibliography}{10}

\bibitem{BGL2014}
D.~Bakry, I.~Gentil, and M.~Ledoux.
\newblock {\em Analysis and Geometry of {Markov} Diffusion Operators}.
\newblock Grundlehren der mathematischen Wissenschaften 348. A Series of
  Comprehensive Studies in Mathematics. Springer, Cham, 2014.

\bibitem{BGPW2014}
A.~Bendikov, A.~Grigor'yan, C.~Pittet, and W.~Woess.
\newblock Isotropic {Markov} semigroups on ultra-metric spaces.
\newblock {\em Russian Mathematical Surveys}, 69:589--680, 2014.

\bibitem{brad_HeatMumf}
P.E. Bradley.
\newblock Heat equations and wavelets on {Mumford} curves and their finite
  quotients.
\newblock {\em Journal of Fourier Analysis and Applications}, 29(5):62, 2023.

\bibitem{SchottkyDiff}
P.E. Bradley.
\newblock Schottky-invariant $p$-adic diffusion operators.
\newblock {\em Journal of Fourier Analysis and Applications}, 31(1):8, 2025.

\bibitem{brad_thetaDiffusionTateCurve}
P.E. Bradley.
\newblock Theta-induced diffusion on {Tate} elliptic curves over
  non-archimedean local fields.
\newblock {\em Pacific Journal of Mathematics}, 334(1):13--42, 2025.

\bibitem{divgrad_p}
P.E. Bradley.
\newblock Topological applications of $p$-adic divergence and gradient
  operators.
\newblock {\em Journal of Mathematical Physics}, 66:021504, 2025.

\bibitem{NonAutonomousDiffusion}
P.E. Bradley and \'A.~Mor\'an Ledezma.
\newblock A non-autonomous $p$-adic diffusion equation on time changing graphs.
\newblock {\em Reports on Mathematical Physics}, in press.
\newblock arXiv:2407:21555 [math.AP].

\bibitem{BL-TopoIndex_p}
P.E. Bradley and \'A.M. Ledezma.
\newblock Approximating diffusion on finite multi-topology systems using
  ultrametrics.
\newblock 2411:arXiv:00806 [cs.DM], 2024.

\bibitem{BL_shapes_p}
P.E. Bradley and \'A. {Mor\'an Ledezma}.
\newblock Hearing shapes with $p$-adic {Laplacians}.
\newblock {\em J. Math. Phys.}, 64:113502, 2023.

\bibitem{Ding2002}
Y.~Ding.
\newblock Heat kernels and {Green}'s functions on limit spaces.
\newblock {\em Communications in Analysis and Geometry}, 10(3):475--514, 2002.

\bibitem{EN2006}
K.-J. Engel and R.~Nagel.
\newblock {\em One-Parameter Semigroups for Linear Evolution Equations}.
\newblock Springer, New York, 2000.

\bibitem{evans}
L.C. Evans.
\newblock {\em Partial Differential Equations}, volume~19 of {\em Graduate
  Studies in Mathematics}.
\newblock American Mathematical Society, Rhode Island, second edition, 2010.

\bibitem{GK2015}
P.~Górka and T.~Kostrzewa.
\newblock Sobolev spaces on metrizable groups.
\newblock {\em Ann. Acad. Sci. Fenn. Math.}, 40:837--849, 2015.

\bibitem{GK2020}
P.~Górka and T.~Kostrzewa.
\newblock A second look of {Sobolev} spaces on metrizable groups.
\newblock {\em Ann. Acad. Sci. Fenn. Math.}, 45:95--120, 2020.

\bibitem{GKR2014}
P.~Górka, T.~Kostrzewa, and E.G. Reyes.
\newblock Sobolev spaces on locally compact {Abelian} groups: compact
  embeddings and local spaces.
\newblock {\em J. Function Spaces}, page Article 404738, 2014.

\bibitem{Kim2010}
Y.-C. Kim.
\newblock A simple proof of the $p$-adic version of the {Sobolev} embedding
  theorem.
\newblock {\em Commun. Korean Math. Soc.}, 25(1):27--36, 2010.

\bibitem{Kochubei2023}
A.N. Kochubei.
\newblock The {Vladimirov-Taibleson} operator: Inequalities, {Dirichlet}
  problem, boundary {Hölder} regularity.
\newblock {\em J. Pseudo-Differ. Oper. Appl.}, 14(31), 2023.

\bibitem{LQ2016}
Y.~Li and H.~Qiu.
\newblock $p$-adic {Laplacian} in local fields.
\newblock {\em Nonlinear Analysis}, 139:131--151, 2016.

\bibitem{PRSWY2024}
T.~Pierce, R.~Rajkumar, A.~Stine, D.~Weisbart, and A.M. Yassine.
\newblock Brownian motion in a vector space over a local field is a scaling
  limit.
\newblock {\em Expositiones Mathematicae}, 42(6):125607, 2024.

\bibitem{RW2023}
R.~Rajkumar and D.~Weisbart.
\newblock Components and exit times of {Brownian} motion in two or more
  $p$-adic dimensions.
\newblock {\em Journal of Fourier Analysis and Applications}, 29:75, 2023.

\bibitem{RZ2010}
J.J. {Rodr\'{i}guez}-Vega and W.A. {Z\'{u}\~{n}iga}-Galindo.
\newblock Elliptic pseudodifferential equations and {Sobolev} spaces over
  $p$-adic fields.
\newblock {\em Pacific Journal of Mathematics}, 246(2):407--420, 2010.

\bibitem{Taibleson1975}
M.H. Taibleson.
\newblock {\em Fourier analysis on local fields}.
\newblock Princeton University Press, Princeton, N.J., University of Tokyo
  Press, Tokyo, 1975.

\bibitem{TAB2024}
A.~Torresblanca-Badillo, E.~Arroyo-Ortiz, and R.~Barrios-Garizao.
\newblock Pseudo-differential operators in several $p$-adic variables and
  sub-{Markovian} semigroups.
\newblock {\em J. Pseudo-Differ. Oper. Appl.}, 15:51, 2024.

\bibitem{Kampen}
{N.G.} {van Kampen}.
\newblock {\em Stochastic Processes in Physics and Chemistry}.
\newblock North-Holland Personal Library, 3rd edition, 2007.

\bibitem{EngelGroup}
J.P. Velasquez-Rodriguez.
\newblock The spectrum of the {Vladimirov sub-Laplacian on the compact Engel}
  group.
\newblock arXiv:2407.06289 [math.RT], 2024.

\bibitem{UnitaryDual_p}
J.P. Velasquez-Rodriguez.
\newblock Unitary dual and matrix coefficients of compact nilpotent $p$-adic
  {Lie} groups with dimension $d\le5$.
\newblock {\em Bol. Soc. Mat. Mex.}, 31:37, 2025.

\bibitem{Weisbart2024}
D.~Weisbart.
\newblock $p$-adic {Brownian} motion is a scaling limit.
\newblock {\em J. Phys. A: Math. Theor.}, 57:205203, 2024.

\bibitem{ZunigaNetworks}
W.~{Z\'{u}\~{n}iga-Galindo}.
\newblock Reaction-diffusion equations on complex networks and {Turing}
  patterns, via $p$-adic analysis.
\newblock {\em Journal of Mathematical Analysis and Applications},
  491(1):124239, 2020.

\bibitem{ZunigaZeta2017}
W.A. {Z\'{u}\~{n}niga}-Galindo.
\newblock Local zeta functions, pseudodifferential operators and {Sobolev}-type
  spaces over non-{Archimedean} local fields.
\newblock {\em $p$-Adic Numbers, Ultrametric Analysis and Applications},
  9(4):314--335, 2017.

\bibitem{Zuniga2017}
W.A. Zúñiga-Galindo.
\newblock Non-{Archimedean} white noise, pseudodifferential stochastic
  equations, and massive {Euclidean} fields.
\newblock {\em Fourier Anal Appl}, 23:288--323, 2017.

\end{thebibliography}

\end{document}